\documentclass[reqno]{amsart}
\usepackage[T1]{fontenc}
\usepackage{dsfont}
\usepackage{mathrsfs}
\usepackage[colorlinks, citecolor=blue, linkcolor=blue]{hyperref}
\usepackage{xcolor}
\usepackage[a4paper,asymmetric]{geometry}
\usepackage{mathscinet}
\usepackage{latexsym}
\usepackage{amsthm}
\usepackage{amssymb}
\usepackage{amsfonts}
\usepackage{amsmath}
\usepackage{longtable}
\usepackage{graphicx}
\usepackage{multirow}
\usepackage{multicol}
\usepackage{amsfonts, amsmath}
\usepackage{latexsym,bm,amsfonts,amssymb,pifont,mathbbol,bbm}
\usepackage{verbatim}
\newtheorem{theorem}{Theorem}[section]
\newtheorem{thm}[theorem]{Theorem}
\newtheorem{lemma}[theorem]{Lemma}
\newtheorem{lem}[theorem]{Lemma}
\newtheorem{remark}[theorem]{Remark}
\newtheorem{proposition}[theorem]{Proposition}

\newtheorem{corollary}[theorem]{Corollary}
\newtheorem{hyp}[theorem]{HYPOTHESIS}
\theoremstyle{definition}

\newtheorem{defn}[theorem]{Definition}

\theoremstyle{remark}
\numberwithin{equation}{section}

 \DeclareMathAlphabet{\mathpzc}{OT1}{pzc}{m}{it}
 \DeclareMathAlphabet{\mathsfsl}{OT1}{cmss}{m}{sl}

\newcommand{\abs}[1]{\left\vert#1\right\vert}

\newcommand{\Be}{\begin{equation}}
\newcommand{\Ee}{\end{equation}}
\newcommand{\Bs}{\begin{split}}
\newcommand{\Es}{\end{split}}
\newcommand{\Bes}{\begin{equation*}}
\newcommand{\Ees}{\end{equation*}}
\newcommand{\BT}{\begin{thm}}
\newcommand{\ET}{\end{thm}}
\newcommand{\Bp}{\begin{proof}}
\newcommand{\Ep}{\end{proof}}
\newcommand{\BL}{\begin{lem}}
\newcommand{\EL}{\end{lem}}
\newcommand{\BP}{\begin{proposition}}
\newcommand{\EP}{\end{proposition}}
\newcommand{\BC}{\begin{corollary}}
\newcommand{\EC}{\end{corollary}}
\newcommand{\BR}{\begin{remark}}
\newcommand{\ER}{\end{remark}}
\newcommand{\BD}{\begin{defn}}
\newcommand{\ED}{\end{defn}}
\newcommand{\BI}{\begin{itemize}}
\newcommand{\EI}{\end{itemize}}

\allowdisplaybreaks

\begin{document}
\title[The limit of mesoscopic one-way fluxes in a nonequilibrium]{The limit of mesoscopic one-way fluxes in a nonequilibrium chemical reaction with complex mechanism}
\author{Yong Chen}
 \address{School of Mathematics and Statistics, Jiangxi Normal University, Nanchang, 330022,  China}
 \email{zhishi@pku.org.cn}
 \author{Wu-Jun Gao}
 \address{College of Big Data and Internet, Shenzhen Technology University, Shenzhen, 518118,  China}
 \email {gaowujun@sztu.edu.cn}
\author{Lu-Yao Qiao}
 \address{School of Mathematics and Statistics, Jiangxi Normal University, Nanchang, 330022,  China}
\email{1436858457@qq.com, (Corresponding author.)}
 \author{Ying Sheng}
 \address{School of Mathematics and Statistics, Jiangxi Normal University, Nanchang, 330022, China }
\email{ 2923839569@qq.com}
\begin{abstract}
Peng et al. (2020) formulated one-way fluxes for a general chemical reaction far from equilibrium, with arbitrary complex mechanisms, multiple intermediates, and internal kinetic cycles. Species are classified into internal $Y's$ and external $X's$. They defined the limit of mesoscopic one-way fluxes when the volume of the tank reactor tends to infinity as macroscopic one-way fluxes,  but a rigorous proof of existence of the limit is still awaiting.
In this article, we fill this gap under a mild hypothesis: the Markov chain associated with the chemical master equation has finite states and any two rows in the stoichiometric matrices for the internal species $Y's$ are not identical. In fact, an explicit expression of the limit is obtained.
\end{abstract}
\maketitle

\section{ Introduction}\label{sec 03}

Nonlinear and nonequilibrium systems of biochemical reactions occur widely and generally in various living biochemical systems\cite{ref10, ref13, ref22}, such as the stochastic kinetics of enzyme\cite{ref17, ref14},  phosphorylation-dephosphorylation kinetics of proteins\cite{ref20}, Calcium signal dynamics\cite{ref19}, and stochastic state transitions of cell
populations \cite{ref18}. For years, it's believed that there exists a fundamental relationship between unbalanced kinetic one-way fluxes and thermodynamic chemical driving forces in nonequilibrium chemical reaction systems. But the formula for this fundamental relation and the very definition of kinetic one-way fluxes in nonlinear chemical reaction networks, were missing for a long time. Peng et al. (2020) \cite{ref1} considered the general chemical reaction network (CRN) in a nonequilibrium steady state (NESS) sustained by a chemostat \cite{ref23}.  They defined the one-way flux of each reaction cycle through the cycle fluxes in the counting space of the corresponding chemical master equation model \cite{ref24} describing the stochastic kinetics of molecular numbers.
Then, they prove the equation
\begin{equation}
    \Delta G=-k_BT\ln (\frac{J^+}{J^-})
\end{equation}
for a general chemical reaction, where $J^+$ and $J^-$ are forward and backward one-way fluxes and $\Delta G$ is the corresponding free energy difference.

In their work, they define the limit of mesoscopic kinetic one-way fluxes dividing $V$ when $V\rightarrow\infty$ as macroscopic kinetic one-way fluxes i.e.,  $$\mathcal{J}_{\tilde{c}}=\lim _{V\rightarrow\infty}\frac{w_{\tilde{c}}}{V}.$$ Here $w_{\tilde{c}}$ is the mesoscopic kinetic one-way fluxes defined as the reaction cycle fluxes, and $V$ is the volume of the tank reactor. However the existence of $\mathcal{J}_{\tilde{c}}$ was not very clear.

In this article, we will give a clear expression of the mesoscopic reaction cycle flux $w_{\tilde{c}}$ in Peng et al. (2020) \cite{ref1} based on general mathematical results on cycle fluxes of a Markov process\cite{ref25}\cite{ref26}. Then, we show existence of the limit of macroscopic kinetic one-way flux and give a explicit expression by rearrangement cycle $c$. 
\section{Preliminary}
\subsection{Chemical Reaction Model and Symbols}\label{subsec model}\quad
In this article, we use the model from Peng et al. (2020) \cite{ref1} below.
A general CRN in a continuously stirred tank reactor consists of a set of species $X_1, X_2,..., X_{N_1} ,Y_1,Y_2,...,Y_{N_2}$, and a set of $M$ reactions between them $R_1, R_2,..., R_M$.  Species can be further classified into internal $Y's$ and external $X's$.
A closed CRN has $N_1 = 0$.

The $M$ reactions, including those between the internal species and chemostated ones can be classified into three groups:
\begin{equation}
\begin{array}{l}
R_{1}: \sum_{j=1}^{N_{2}} \Xi_{+1}^{j,  Y} Y_{j}+\sum_{i=1}^{N_{1}} \Xi_{+1}^{i,  X} X_{i} \rightleftharpoons \sum_{j=1}^{N_{2}} \Xi_{-1}^{j,  Y} Y_{j} \\
\\
R_{\ell}: \sum_{j=1}^{N_{2}} \Xi_{+\ell}^{j,  Y} Y_{j} \rightleftharpoons \sum_{j=1}^{N_{2}} \Xi_{-\ell}^{j,  Y} Y_{j} \\
\\
R_{M}: \sum_{j=1}^{N_{2}} \Xi_{+M}^{j,  Y} Y_{j} \rightleftharpoons \sum_{i=1}^{N_{1}} \Xi_{-M}^{i,  X} X_{i}+\sum_{j=1}^{N_{2}} \Xi_{-M}^{j,  Y} Y_{j}
\end{array}
\end{equation}
where $\ell=2, 3, . . . , M-1$. The symbols $\Xi_{+1}^{i, X}$ and $\Xi_{-M}^{i, X}$ above is denoted as the coefficient of $i$th $X's$ species in the $1$th forward or $M$th backward reaction. Similary, $\Xi_{+\ell}^{j, Y}$ and $\Xi_{-\ell}^{j, Y}$ above are denoted as the coefficient of $j$th $Y's$ species in the $\ell$th forward and backward reaction, where $\ell=1, 2, . . . , M$.

Here only the reactions $R_1$ and $R_M$ exchange materials between $X$ and $Y$. Obviously, the reaction system has a single input complex and a single output complex. All the $Y's$ species are intermediates in the mechanistic details of the overall reaction,  transforming $\sum_{i=1}^{N_1}\Xi_{+1}^{i, X}X_i$
to $\sum_{i=1}^{N_1}\Xi_{-M}^{i, X}X_i$.

The stoichiometric matrices $\Xi^X$ and $\Xi^Y$ for $X's$ and $Y's$ are, respectively,
\begin{equation}
\begin{aligned}
\Xi^{X} &=\left\{\Xi_{i \ell}^{X}=\Xi_{-\ell}^{i,  X}-\Xi_{+\ell}^{i,  X}\right\}_{N_{1} \times M} ,\\
\Xi^{Y} &=\left\{\Xi_{j \ell}^{Y}=\Xi_{-\ell}^{j,  Y}-\Xi_{+\ell}^{j,  Y}\right\}_{N_{2} \times M}.
\end{aligned}
\end{equation} 
Note that for the matrix $\Xi^X$, only the first and the last row have nonzero entries.Denote by $\Xi_{\ell}^{Y}$ the $\ell$-th column of the matrix $\Xi^Y$ with $1\leq\ell\leq M$.

The complete stoichiometric matrix, then, is $$\Xi=\begin{bmatrix} {\Xi^X}\\ {\Xi^Y} \end{bmatrix}.$$

The concentrations of all the species $X's$ are clamped at constant levels. All reactions in this system are stochastic, elementary, and reversible, under isothermal and isobaric conditions with fixed volume $V$.
\subsection{The Transition Trajectory of Chemical reactions}\quad

On the mesoscopic scale, we focus on molecular numbers in the CRN,  which are stochastic.
 Stochastic vector $\textbf{n}^{\textbf{Y}}(t) = [n^Y_1(t),  n^Y_2(t), . . . ,  n^Y_{N_2}(t)]$ is used to denote molecular numbers of the internal species at time $t$.  $X's$ species still possess fixed molecular numbers,  denoted by $\textbf{n}^{\textbf{X}}(t) = [n^X_1(t),  n^X_2(t), . . . ,  n^X_{N_1}(t)]$.
The evolution of $\textbf{n}^{\textbf{Y}}(t)$ can be seen as a continuous time Markov chain on a high-dimensional graph (called a counting space).

In the present paper, we assume the Markov chain satisfying the following hypothesis:
\begin{hyp}\label{hypthe 1}
(\romannumeral1) The Markov chain $\textbf{n}^{\textbf{Y}}(t) $ is irreducible 
 with a finite state space $S$. \\
(\romannumeral2) Any two columns in the matrix $\Xi$ are not identical.
\end{hyp}
The hypothesis $(i)$ is just for the mathematical simplicity. The  hypothesis $(ii)$ is very mild and natural since the exsitence of two identical columns of the matrix $\Xi$ means that the two chemical reactions associated with that columns are completely equivalent and can not be distinguished. We would like to point out that in Peng et al. (2020) \cite{ref1}, both the above two hypotheses are assumed implicitly.  We will get rid of these hypotheses in a separate paper.

Denote by $N$ the number of elements of the state space $S$. Furthermore, denote by $Q=(q_{ij})_{N\times N}$ the transition rate matrix of the Markov chain.
Each state in $S$, namely each vertex of the graph, has an $N_2$-dimensional coordinate $\textbf{n}^{\textbf{y}}= [n^y_1,  n^y_2, . . . ,  n^y_{N_2}]$. The occurrence of an elementary reaction in $\{R_1, . . . , R_M\}$ converts state of the system from one vertex to another.

 According to the Markov processes theory in \cite{ref23}, the counting space is a scaffold for a Markov process with transition rates
\small
\begin{equation}\label{r1}
    r_{+\ell}\left(\mathbf{n}^{\mathbf{y}} ; V\right) \triangleq
\tilde{k}_{+\ell}(V)
\prod_{j=1}^{N_{2}} n_{j}^{y}\left(n_{j}^{y}-1\right) \cdots\left(n_{j}^{y}-\Xi_{+\ell}^{j,  Y}+1\right)
\end{equation}
for the forward reaction $R_{+\ell}$,
\begin{equation}\label{r2}
    r_{-\ell}\left(\mathbf{n}^{\mathbf{y}} ; V\right) \triangleq
\tilde{k}_{-\ell}(V)
\prod_{j=1}^{N_{2}} n_{j}^{y}\left(n_{j}^{y}-1\right) \cdots\left(n_{j}^{y}-\Xi_{-\ell}^{j,  Y}+1\right)
\end{equation}
for the corresponding backward reaction $R_{-\ell}$,  except two reactions involving $X$,  i.e.,
\begin{equation}\label{r3}
\begin{aligned}
    r_{-M}\left(\mathbf{n}^{\mathbf{x}}, \mathbf{n}^{\mathbf{y}} ; V\right)\triangleq &\tilde{k}_{-M} \prod_{i=1}^{N_{1}} n_{i}^{x}\left(n_{i}^{x}-1\right) \cdots\left(n_{i}^{x}-\Xi_{-M}^{i,  X}-M\right)\prod_{j=1}^{N_{2}} n_{j}^{y}\left(n_{j}^{y}-1\right) \cdots\left(n_{j}^{y}-\Xi_{-M}^{j,  Y}-M\right)
\end{aligned}
\end{equation}
and
\begin{equation}\label{r4}
\begin{aligned}
    r_{+1}\left(\mathbf{n}^{\mathbf{x}}, \mathbf{n}^{\mathbf{y}} ; V\right)\triangleq &\tilde{k}_{+1} \prod_{i=1}^{N_{1}} n_{i}^{x}\left(n_{i}^{x}-1\right) \cdots\left(n_{i}^{x}-\Xi_{+1}^{i,  X}+1\right)\prod_{j=1}^{N_{2}} n_{j}^{y}\left(n_{j}^{y}-1\right) \cdots\left(n_{j}^{y}-\Xi_{+1}^{j,  Y}+1\right)
\end{aligned}
\end{equation}

 Here, $\tilde{k}_{\pm \ell}$ are called the forward and backward  macroscopic rate constants in the stochastic model, and $k_{\pm \ell}=\tilde{k}_{\pm \ell}V^{n_{\pm\ell}-1}$ are mesoscopic rate constants, where $n_{\pm\ell}$ is the summation of stoichiometric numbers of reactants in the reaction $R_{\pm\ell}$
 $$n_{+\ell}=\sum_{j=1}^{N_{2}}\Xi_{+\ell}^{j,  Y} Y_{j},\ell=2,3,...,M$$
 $$n_{-\ell}=\sum_{j=1}^{N_{2}}\Xi_{-\ell}^{j,  Y} Y_{j},\ell=1,2,...,M-1$$
 and
 $$n_{+1}=\sum_{j=1}^{N_{2}} \Xi_{+1}^{j,  Y} Y_{j}+\sum_{i=1}^{N_{1}} \Xi_{+1}^{i,  X} X_{i}$$ 
  $$n_{-M}=\sum_{j=1}^{N_{2}} \Xi_{-M}^{j,  Y} Y_{j}+\sum_{i=1}^{N_{1}} \Xi_{-M}^{i,  X} X_{i}.$$ 

 In our model, the reaction coefficients decide all the step size along the transition trajectory. When hypothesis (\ref{hypthe 1} \romannumeral2) satisfied, the phenomenon that two different reactions in ${R_1, . . . , R_M}$ leading to same transfering in the state space $S$ will be removed. Then, the transition rate of the Markov chain is either,
\begin{equation}
    \begin{aligned}
    &q_{\textbf{n}^{y}}(\textbf{n}^{y}+\Xi_{\ell}^Y)=r_{+\ell}(\textbf{n}^{y};V), (2\leq\ell\leq M), \\
    &q_{\textbf{n}^{y}}(\textbf{n}^{y}-\Xi_{\ell}^Y)=r_{-\ell}(\textbf{n}^{y};V), (1\leq\ell\leq M-1),
   \end{aligned}
\end{equation}
    or
\begin{equation}
    \begin{aligned}
    &q_{\textbf{n}^{y}}(\textbf{n}^{y}+\Xi_{1}^Y)=r_{+1}(\textbf{n}^{x}, \textbf{n}^{y};V)\\
    &q_{\textbf{n}^{y}}(\textbf{n}^{y}-\Xi_{M}^Y)=r_{-M}(\textbf{n}^{x}, \textbf{n}^{y};V).
    \label{traneq}
    \end{aligned}
\end{equation}

Then the chemical master equation (CME) describing the evolution of the probability $$p_{V}\left(\mathbf{n}^{\mathbf{y}}, t\right)=\operatorname{Prob}\left(\mathbf{n}^{\mathbf{Y}}(t)=\mathbf{n}^{\mathbf{y}}\right)$$ is
$$
\begin{aligned}
\frac{\partial p_{V}\left(\mathbf{n}^{\mathbf{y}},  t\right)}{\partial t}=& \sum_{\ell=2}^{M-1}\left\{r_{+\ell}\left(\mathbf{n}^{\mathbf{y}}-\Xi_{\ell}^{Y} ; V\right) p_{V}\left(\mathbf{n}^{\mathbf{y}}-\Xi_{\ell}^{Y},  t\right)\right. -\left[r_{+\ell}\left(\mathbf{n}^{\mathbf{y}} ; V\right)+r_{-\ell}\left(\mathbf{n}^{\mathbf{y}} ; V\right)\right] p_{V}\left(\mathbf{n}^{\mathbf{y}},  t\right) \\
&\left. +r_{-\ell}\left(\mathbf{n}^{\mathbf{y}}+\Xi_{\ell}^{Y} ; V\right) p_{V}\left(\mathbf{n}^{\mathbf{y}}+\Xi_{\ell}^{Y},  t\right)\right\} +\left\{r_{+1}\left(\mathbf{n}^{\mathbf{x}},  \mathbf{n}^{\mathbf{y}}-\Xi_{1}^{Y} ; V\right) p_{V}\left(\mathbf{n}^{\mathbf{y}}-\Xi_{1}^{Y},  t\right)\right. \\
&-\left[r_{+1}\left(\mathbf{n}^{\mathbf{x}},  \mathbf{n}^{\mathbf{y}} ; V\right)+r_{-1}\left(\mathbf{n}^{\mathbf{y}} ; V\right)\right] p_{V}\left(\mathbf{n}^{\mathbf{y}},  t\right)
\left. +r_{-1}\left(\mathbf{n}^{\mathbf{y}}+\Xi_{1}^{Y} ; V\right) p_{V}\left(\mathbf{n}^{\mathbf{y}}+\Xi_{1}^{Y},  t\right)\right\} \\
&+\left\{r_{+M}\left(\mathbf{n}^{\mathbf{y}}-\Xi_{M}^{Y} ; V\right) p_{V}\left(\mathbf{n}^{\mathbf{y}}-\Xi_{M}^{Y},  t\right)\right.
-\left[r_{+M}\left(\mathbf{n}^{\mathbf{y}} ; V\right)+r_{-M}\left(\mathbf{n}^{\mathbf{x}},  \mathbf{n}^{\mathbf{y}} ; V\right)\right] p_{V}\left(\mathbf{n}^{\mathbf{y}},  t\right) \\
&\left. +r_{-M}\left(\mathbf{n}^{\mathbf{x}},  \mathbf{n}^{\mathbf{y}}+\Xi_{M}^{Y} ; V\right) p_{V}\left(\mathbf{n}^{\mathbf{y}}+\Xi_{M}^{Y},  t\right)\right\}.
\end{aligned}
$$

\subsection{Cycle Flux}\quad

 A directed cycle in the counting space is
 a slice of path with the same origin and destination in which no other states overlap.
 It can be expressed in this way: $c = [y_1,  y_2, . . . ,  y_s]$
with some integer $s$ where $y_1,  y_2, . . . ,  y_s\in S$ are successive states in the path which are different from each other.
The next state of $y_s$ in the path is exactly $y_1$,
forming a single closed loop.

We denote an ordered sequence of distinct points $y_1, y_2, . . . , y_s$ by $[y_1, y_2, . . . , y_s]$.
The set $[S]$ is all the finite ordered  sequences $[y_1, y_2. . . , y_s]$ with $s\neq1$.
For any $i$ with indices modulo $k$, we identify $[y_1,  y_2, . . . ,  y_k]$  and $[y_i,  y_{i+1}, . . . ,  y_{i+k-1}]$ as the same directed cycle.

Next, we describe the cycle flux of a continuous time Markov chain, 
please refer to \cite{ref25, ref26, ref28} for details. 
\begin{defn}\label{thm strong}
Suppose $\xi$ is a continuous time Markov chain. 
When Hypothesis (\ref{hypthe 1}(\romannumeral1)) is satisfied,
set
\begin{equation}
    \begin{aligned}
    &T_0(\omega)=0\\
    &T_1(\omega)=\inf\{t>0:\xi_t(\omega)\neq\xi_0(\omega)\}\\
    &T_{r+1}(\omega)=\inf\{t>T_r(\omega):\xi_t(\omega)\neq\xi_{T_r(\omega)}(\omega)\}, \forall r\in \mathbb{N}
    \end{aligned}
    \notag
\end{equation}
where $\mathbb{N}$ is the set of natural numbers.
Define
\begin{equation}
    m_t(\omega)=\sup\{m\geq0:T_m(\omega)\neq t\}
\end{equation}

For such a cycle $c=[y_1,  y_2, . . . ,  y_s]\in [S]$,\\
(\romannumeral1) The number of the cycle $c$ up to time $t$ along the sample path $\{\xi_{m}(\omega)\}_{m\neq0}$ is
\begin{align*}
w_{c, t}=\sum_{l=1}^{m_t(\omega)}\mathbb{1}_{\cup_{k=1}^s\{(\tilde{\omega}):\eta_{l-1}(\tilde{\omega})=[\eta_{l}(\tilde{\omega}), [y_k,  y_{k+1}, . . . ,  y_{k+s-1}]]\}}(\omega)
\end{align*}
where $\mathbb{1}_A(\cdot)$ is the indicator of the set $A$ and the sums $k+1, k+2, . . . , k+s-1$ are understood to be modulo $s$.

(\romannumeral2) The limit $$\lim_{t\rightarrow\infty}\frac{w_{c, t}}{t}$$ exists. It's called the cycle flux for $c$ and denoted by $\omega_{c}$. Moreover, it can be expressed more precisely:
\begin{align*}\label{w_c}
\omega_{c}=(-1)^{s-1} q_{y_{1} y_{2}} q_{y_{2} y_{3}} \ldots q_{y_{s} y_{1}} \frac{\left|Q\left(\left\{y_{1},  y_{2},  \ldots,  y_{s}\right\}^{\mathrm{c}}\right)\right|}{\sum_{j \in S}\left|Q\left(\{j\}^{\mathrm{c}}\right)\right|},
\end{align*}
where $\abs{Q(H)}$ is the determinant of $Q$ with rows and columns indexed by $H$ and $H^{\mathrm{c}}$ is the complement of $H$.
\end{defn}      

\subsection{Reaction Cycle and Reaction Cycle Flux}\label{rcm}\quad
\begin{defn}[Peng et al.\cite{ref1}]
Denote $\mathcal{C}_{\infty}$ as the set of all cycles in the counting space. For any cycle $c=[y_1,y_2,...,y_s] \in \mathcal{C}_{\infty}$, the net numbers of occurrence of all the reactions in $c$ satisfies the following linear equation: $$ \Xi^Y \tilde{c}'=0.$$
(\romannumeral1)
A map $ \phi: \mathcal{C}_{\infty}\rightarrow\mathbb{R}^M$ is given by $$\phi(c)=\tilde{c}=(c_1, c_2, . . . , c_M),$$ 
where $c_\ell,1\leq\ell\leq M,$ is the number of occurrence of $\ell$-th forward reaction minus the number of occurrence backward reaction in the cycle $c$.\\
(\romannumeral2)
The cycle flux for $\tilde{c}$ is defined as: 
\begin{align}
    \omega_{\tilde{c}}=\sum_{\phi(c)=\tilde{c}}\omega_{c},
\end{align}which is the averaged frequency $\tilde{c}$ occurs in the CME model.
\end{defn}

\section{Main Results and their Proofs}
 In this section,  we show the existence of the macroscopic reaction cycle flux $\mathcal{J}_{\tilde{c}}$, i.e. the existence of the limit of mesoscopic reaction flux $$\lim_{V\to \infty}\frac{1}{V}\sum_{\phi(c)=\tilde{c}}w_c.$$
 
 Lemma~\ref{w_c1}  gives a more clear presentation of the mesoscopic reaction flux $w_{\tilde{c}}$.
\begin{lemma}\label{w_c1}
\begin{equation}
    \begin{aligned}
\omega_{\tilde{c}}=
\sum_{s=2}^N\sum_{
\tiny{
\substack{
     c\in [S]\\c=[y_1, y_2, . . . y_s]}
}
}
\left(w_c
\cdot
\prod_{l=1}^M\mathbbm{1}_{\left\{\sum_{h=1}^{s}
\big(\mathbbm{1}_{\{y_{h+1}=y_h+\Xi_l^Y\}}-
\mathbbm{1}_{\{y_{h+1}=y_h-\Xi_l^Y\}}\big)=c_l\right\}}\right)
\end{aligned}
\end{equation}
\end{lemma}
\begin{proof}
First, we pick out the cycle satisfying the map $\phi(c)=\tilde{c}$ in the set $[S]$ with different cycle length. Next, for any cycle $c=[y_1,y_2,...,y_s]\in[S]$, the number of occurrence of $\ell$-th forward reaction is $\sum_{h=1}^{s}\mathbbm{1}_{\{y_{h+1}=y_h+\Xi_l^Y\}}$, and the number of occurrence of $\ell$-th backward reaction is $\sum_{h=1}^{s}\mathbbm{1}_{\{y_{h+1}=y_h-\Xi_l^Y\}}$, $\ell=1,2,...,M$. Thus, the net number of occurrence of $\ell$-th reaction in $c$ is $$\sum_{h=1}^{s}(\mathbbm{1}_{\{y_{h+1}=y_h+\Xi_l^Y\}}-\mathbbm{1}_{\{y_{h+1}=y_h-\Xi_l^Y\}}).$$
 When all sums are equal to corresponding $c_\ell,\ell=1,2,...,M$, it's obvious that the net numbers of occurrence of all the reactions in $c$ form a reaction cycle $\tilde{c}=(c_1,c_2,...,c_M)$.
\end{proof}
\begin{lemma}\label{rate}
Take $V\rightarrow\infty$, all reaction rates become
\begin{equation}
    \begin{aligned}
    &r_{+\ell}=V\left(k_{+\ell}N_A\prod_{j=1}^{N_2}(\beta_j)^{\Xi_{+\ell}^{j,  Y}}+O(V^{-1})\right),\quad\ell=2,3,\cdots,M
    \\
    &r_{-\ell}=V\left(k_{-\ell}N_A\prod_{j=1}^{N_2}(\beta_j)^{\Xi_{-\ell}^{j,  Y}}+O(V^{-1})\right),\quad\ell=1,2,\cdots,M-1
\\
or\quad\quad&
\\
    &r_{+1}=V\left(k_{+1}N_A\prod_{i=1}^{N_1}(\alpha_i)^{\Xi_{+1}^{i,  X}}\prod_{j=1}^{N_2}(\beta_j)^{\Xi_{+1}^{j,  Y}}+O(V^{-1})\right)
    \\
    &r_{-M}=V\left(k_{-M}N_A\prod_{i=1}^{N_1}(\alpha_i)^{\Xi_{-M}^{i,  X}}\prod_{j=1}^{N_2}(\beta_j)^{\Xi_{-M}^{j, Y}}+O(V^{-1})\right)
    \\
    \end{aligned}
\end{equation}
\end{lemma}
Please refer to Kurtz \cite{ref2} for this lemma. We give a proof for the readers' convenience.
\begin{proof} Let $N_A$ be the Avogadro's number.
First, 
notice that $$k_{\pm\ell}=\tilde{k}_{\pm \ell}V^{n_{\pm\ell}-1},$$ and then rewrite Eq. (\ref{r1})-(\ref{r4}) in terms of concentrations $$\beta_i=n_i^y/(VN_A).$$
Notice $n_{+\ell}=\sum_{j=1}^{N_2}\Xi_{+\ell}^{j,Y}$,then
\begin{align*}
r_{+\ell}\left(\beta; V\right)=&\frac{k_{+\ell}}{(VN_A)^{\sum_{j=1}^{N_2}\Xi_{+\ell}^{j,  Y}-1}}\prod_{j=1}^{N_{2}}\beta_{j} (VN_A)\left(\beta_{j}(VN_A)-1\right)
\cdots\left(\beta_{j}(VN_A)-\Xi_{+\ell}^{j,  Y}+1\right)
\\
=&k_{+\ell}\cdot
\{
(VN_A)\prod_{j=1}^{N_{2}}(\beta_{j})^{\Xi_{+\ell}^{j,  Y}}
-(\prod_{j=1}^{N_{2}}(\beta_{j})^{\Xi_{+\ell}^{j,  Y}-1})(\sum_{\mu=1}^{\Xi_{+\ell}^{j,  Y}-1}\mu)+(VN_A)^{-1}(\prod_{j=1}^{N_{2}}(\beta_{j})^{\Xi_{+\ell}^{j,  Y}-2})
\\
&\cdot(\sum_{\substack{\mu>\nu\\\mu, \nu=1, 2, \cdots, \Xi_{+\ell}^{j,  Y}-1}}\mu\nu)
+\cdots +(-1)^{\Xi_{+\ell}^{j,  Y}}(VN_A)^{(-\sum_{j=1}^{N_2}\Xi_{+\ell}^{j,  Y}+1)}\prod_{\mu=1}^{\Xi_{+\ell}^{j,  Y}-1}\mu
\}
\end{align*}
for the forward reaction $R_{+\ell}$, where $\ell=2, 3, \cdots, M$.

When $V\rightarrow\infty$, the above formula become $$r_{+\ell}=V\left(N_Ak_{+\ell}\prod_{j=1}^{N_2}(\beta_j)^{\Xi_{+\ell}^{j,  Y}}+O(V^{-1})\right),\quad\ell=2, 3, \cdots, M.$$

In the same vein, we have 
 $$r_{-\ell}=V\left(N_Ak_{-\ell}\prod_{j=1}^{N_2}(\beta_j)^{\Xi_{-\ell}^{j,  Y}}+O(V^{-1})\right),\quad\ell=1,2,\cdots,M-1.$$

For the other two transition rates, notice $$n_{+1}=\sum_{i=1}^{N_1}\Xi_{+1}^{i,X}+\sum_{j=1}^{N_2}\Xi_{+1}^{j,Y}$$
and
$$n_{-M}=\sum_{i=1}^{N_1}\Xi_{-M}^{i,X}+\sum_{j=1}^{N_2}\Xi_{-M}^{j,Y}$$
then,
\begin{align*}
&r_{+1}\left(\alpha, \beta; V\right)
\\
=&\frac{k_{+1}}{(VN_A)^{\sum_{i=1}^{N_1}\Xi_{+1}^{i,X}+\sum_{j=1}^{N_2}\Xi_{+1}^{j,Y}-1}}
\prod_{i=1}^{N_{1}} (VN_A)\alpha_{i}\cdot\left(\alpha_{i}(VN_A)-1\right)
\cdots\left(\alpha_{i}(VN_A)-\Xi_{+1}^{i,  X}+1\right)
\\
&\cdot \prod_{j=1}^{N_{2}}\beta_{j} (VN_A)\cdot\left(\beta_{j}(VN_A)-1\right)
\cdots\left(\beta_{j}(VN_A)-\Xi_{+1}^{j,  Y}+1\right)
\\
=&k_{+1}(VN_A)\prod_{i=1}^{N_{1}}\left(\alpha_{i}(\alpha_{i}-\frac{1}{VN_A})
\cdots\left(\alpha_{i}-\frac{\Xi_{+1}^{i,X}-1}{VN_A}\right)\right)
\cdot
\prod_{j=1}^{N_{2}}\left(\beta_{j}(\beta_{j}-\frac{1}{VN_A})
\cdots\left(\beta_{j}-\frac{\Xi_{+1}^{j,Y}-1}{VN_A}\right)\right)
\end{align*}

When $V\rightarrow\infty$, the above formula become $$r_{+1}=V\left(N_Ak_{+1}(\prod_{i=1}^{N_1}\alpha_i)^{\Xi_{+1}^{i,  X}}(\prod_{j=1}^{N_2}\beta_j)^{\Xi_{+1}^{j,Y}}+O(V^{-1})\right).$$

In the same vein, we have  $$r_{-M}=V\left(N_Ak_{-M}(\prod_{i=1}^{N_1}\alpha_i)^{\Xi_{-M}^{i,  X}}(\prod_{j=1}^{N_2}\beta_j)^{\Xi_{-M}^{j, Y}}+O(V^{-1})\right).$$
\end{proof}
Notice that transition rates in lemma~\ref{rate} not only rely on which reaction occurs but also rely on which state the reaction occurs. To represent the transition rate precisely, we define two new reaction cycle similarly to the map in definition~\ref{zfcs}.
\begin{defn}\label{zfcs}
For the cycle $c=[y_1,y_2,...,y_s]\in\mathcal{C}_{\infty}$,\\
(\romannumeral1)
define a map $\phi^{+}:\, \mathcal{C}_{\infty}\rightarrow\mathbb{R}^M$ as  $$\phi^+(c)=\tilde{c}^+=(c_1^+, c_2^+, . . . , c_M^+)',$$ 
 where $c_\ell^+, 1\leq\ell\leq M$ is the number of the occurrence of $\ell$-th forward reaction in cycle $c$. That is to say, $\tilde{c}^+$ is the unique forward reaction cycle associated with $c$.\\
(\romannumeral2)define a map $\phi^{-}:\, \mathcal{C}_{\infty}\rightarrow\mathbb{R}^M$ as  $$\phi^-(c)=\tilde{c}^-=(c_1^-, c_2^-, . . . , c_M^-)',$$ 
where $c_\ell^-, 1\leq\ell\leq M$ is the number of the occurrence of $\ell$-th backward reaction in cycle $c$.
 That is to say, $\tilde{c}^-$ is the unique backward reaction cycle associated with $c$.
\end{defn}
\begin{remark}\label{remarkc}It is easy to check that the following relationships between $$\phi^+(c)=\tilde{c}^+=(c_1^+, c_2^+, . . . , c_M^+)',\,\phi^-(c)=\tilde{c}^-=(c_1^-, c_2^-, . . . , c_M^-)'$$ and $$\phi^(c)=\tilde{c}=(c_1, c_2, . . . , c_M)'$$ holds:
 \begin{enumerate}
	 \item $c_i=c_i^+-c_i^-,c_i^+\geq0,c_i^-\geq0$,where $i=1, 2, . . . , M$.\\
	 \item $\sum_{i=1}^M(c_i^++c_i^-)=s$.
 \end{enumerate}
\end{remark}
\begin{lemma}\label{q_yy}
From lemma (\ref{rate}) and definition (\ref{zfcs}), for cycle $c=[y_1, y_2, \cdots, y_s]$,
\begin{equation}\label{hlbd}
    \begin{aligned}
&q_{y_{1} y_{2}} q_{y_{2} y_{3}} \ldots q_{y_{s} y_{1}}
\\
=&V^s\bigg[N_A^s\cdot\left(k_{+1}^{c_1^+}(\prod_{i=1}^{N_1}(\alpha_i)^{\Xi_{+1}^{i, X}})^{c_1^+}(\prod_{i=1}^{c_1^+}\prod_{j=1}^{N_2}(\beta_{i, j}^{+1})^{\Xi_{+1}^{j,  Y}})\right)
\left(k_{+2}^{c_2^+}(\prod_{i=1}^{c_2^+}\prod_{j=1}^{N_2}(\beta_{i, j}^{+2})^{\Xi_{+2}^{j,  Y}})\right)\cdots
\\
&\left(k_{+M}^{c_M^+}(\prod_{i=1}^{N_1}(\alpha_i)^{\Xi_{+M}^{i,  X}})^{c_M^+}(\prod_{i=1}^{c_M^+}\prod_{j=1}^{N_2}(\beta_{i, j}^{+M})^{\Xi_{+M}^{j,  Y}})\right)\cdot
\left(k_{-1}^{c_1^-}(\prod_{i=1}^{N_1}(\alpha_i)^{\Xi_{-1}^{i, X}})^{c_1^-}(\prod_{i=1}^{c_1^-}\prod_{j=1}^{N_2}(\beta_{i, j}^{-1})^{\Xi_{-1}^{j,  Y}})\right)
\\
&\left(k_{-2}^{c_2^-}(\prod_{i=1}^{c_2^-}\prod_{j=1}^{N_2}(\beta_{i, j}^{-2})^{\Xi_{-2}^{j,  Y}})\right)\cdots
\left(k_{-M}^{c_M^-}(\prod_{i=1}^{N_1}(\alpha_i)^{\Xi_{-M}^{i,  X}})^{c_M^-}(\prod_{i=1}^{c_M^-}\prod_{j=1}^{N_2}(\beta_{i, j}^{-M})^{\Xi_{-M}^{j,  Y}})\right)+O(V^{-1})\bigg]\cdot
    \end{aligned}
\end{equation}
\end{lemma}
\begin{proof}\quad
From definition (\ref{zfcs}), denote the corresponding reaction cycle of $c=[y_1,y_2,\cdots,y_s]$ as $\tilde{c}=[c_1,c_2,...,c_M]$,  forward reaction cycle as  $c^+=[c_1^+,c_2^+,...,c_M^+]$, back reaction cycle as  $c^-=[c_1^-,c_2^-,...,c_M^-]$. There are $c_i^+(i=1, 2, . . . , M)$ states in the cycle on which the $i$-th forward reaction occurs and $c_i^-(i=1, 2, . . . , M)$ states on which the $i$-th backward reaction occurs.  Rearrange the cycle $c$ as an ordered sequence $\hat{c}$,
\begin{equation}
    \hat{c}=(y_{1}^{+1}, y_{2}^{+1}, . . . , y_{c_1^+}^{+1}, . . . , y_{1}^{+M}, y_{2}^{+M}, . . . , y_{c_M^+}^{+M}, y_{1}^{-1}, y_{2}^{-1}, . . . , y_{c_1^-}^{-1}, . . . , y_{1}^{-M}, y_{2}^{-M}, . . . , y_{c_M^-}^{-M})
\end{equation}
where $y_{i}^{\pm\ell}, \ell=1, 2. . . , M$ are the $i$-th state along the cycle $c$ that on which occurs $\ell$-th forward and backward reaction that convert one state to another.
Denote by $\hat{\beta}$ the concentrations vector of $c$ corresponding to $\hat{c}$,
\begin{equation}
    \hat{\beta}=(\beta_{1}^{+1}, \beta_{2}^{+1}, . . . , \beta_{c_1^+}^{+1}, . . . , \beta_{1}^{+M}, \beta_{2}^{+M}, . . . , \beta_{c_M^+}^{+M}, \beta_{1}^{-1}, \beta_{2}^{-1}, . . . , \beta_{c_1^-}^{-1}, . . . , \beta_{1}^{-M}, \beta_{2}^{-M}, . . . , \beta_{c_M^-}^{-M})
\end{equation}
where  all $$\beta_{i}^{\pm\ell}=y_{i}^{\pm\ell}/(VN_A)$$ are $N_2$-dimensional vectors. Furthermore, denote $$\beta_{i, j}^{\pm\ell},\,\, j=1, 2, . . . , N_2$$ the $j$-th element of the vector $\beta_{i}^{\pm\ell}$.
Thus, 
\begin{equation}
\begin{aligned}
&q_{y_{1} y_{2}} q_{y_{2} y_{3}} \ldots q_{y_{s} y_{1}}
\\
=&\left(r_{+1}(\alpha, \beta_{1, 1}^+;V)
r_{+1}(\alpha, \beta_{1, 2}^+;V)\cdots r_{+1}(\alpha, \beta_{1, c_1^+}^+;V)\right)
\left(r_{+2}(\beta_{2, 1}^+;V)r_{+2}(\beta_{2, 2}^+;V)
\cdots
r_{+2}(\beta_{2, c_2^+}^+;V)
\right)
\cdots
\\
&\left(r_{+M}(\alpha, \beta_{M, 1}^+;V)
r_{+M}(\alpha, \beta_{M, 2}^+;V)\cdots r_{+M}(\alpha, \beta_{M, c_M^+}^+;V)\right)\cdot
\\
&\left(r_{-1}(\alpha, \beta_{1, 1}^-;V)
r_{-1}(\alpha, \beta_{1, 2}^-;V)\cdots r_{-1}(\alpha, \beta_{1, c_1^-}^-;V)\right)
\left(r_{-2}(\beta_{2, 1}^-;V)r_{-2}(\beta_{2, 2}^-;V)
\cdots r_{-2}(\beta_{2, c_2^-}^-;V)
\right)\cdots
\\
&\left(r_{-M}(\alpha, \beta_{M, 1}^-;V)
r_{-M}(\alpha, \beta_{M, 2}^-;V)\cdots r_{-M}(\alpha, \beta_{M, c_M^-}^-;V)\right)
\\
=&\left(k_{+1}^{c_1^+}(\prod_{i=1}^{N_1}(\alpha_i)^{\Xi_{+1}^{i, X}})^{c_1^+}(\prod_{i=1}^{c_1^+}\prod_{j=1}^{N_2}(\beta_{i, j}^{+1})^{\Xi_{+1}^{j,  Y}})(VN_A)^{c_1^+}\right)
\left(k_{+2}^{c_2^+}(\prod_{i=1}^{c_2^+}\prod_{j=1}^{N_2}(\beta_{i, j}^{+2})^{\Xi_{+2}^{j,  Y}})(VN_A)^{c_2^+}\right)\cdots
 \\
&\left(k_{+M}^{c_M^+}(\prod_{i=1}^{N_1}(\alpha_i)^{\Xi_{+M}^{i,  X}})^{c_M^+}(\prod_{i=1}^{c_M^+}\prod_{j=1}^{N_2}(\beta_{i, j}^{+M})^{\Xi_{+M}^{j,  Y}})(VN_A)^{c_M^+}\right)\cdot
\\
&\left(k_{-1}^{c_1^-}(\prod_{i=1}^{N_1}(\alpha_i)^{\Xi_{-1}^{i, X}})^{c_1^-}(\prod_{i=1}^{c_1^-}\prod_{j=1}^{N_2}(\beta_{i, j}^{-1})^{\Xi_{-1}^{j,  Y}})(VN_A)^{c_1^-}\right)
\left(k_{-2}^{c_2^-}(\prod_{i=1}^{c_2^-}\prod_{j=1}^{N_2}(\beta_{i, j}^{-2})^{\Xi_{-2}^{j,  Y}})(VN_A)^{c_2^-}\right)\cdots
\\
&\left(k_{-M}^{c_M^-}(\prod_{i=1}^{N_1}(\alpha_i)^{\Xi_{-M}^{i,  X}})^{c_M^-}(\prod_{i=1}^{c_M^-}\prod_{j=1}^{N_2}(\beta_{i, j}^{-M})^{\Xi_{-M}^{j,  Y}})(VN_A)^{c_M^-}\right)+O(V^{s-1})
\\
=&V^s\bigg[N_A^s\cdot\left(k_{+1}^{c_1^+}(\prod_{i=1}^{N_1}(\alpha_i)^{\Xi_{+1}^{i, X}})^{c_1^+}(\prod_{i=1}^{c_1^+}\prod_{j=1}^{N_2}(\beta_{i, j}^{+1})^{\Xi_{+1}^{j,  Y}})\right)
\left(k_{+2}^{c_2^+}(\prod_{i=1}^{c_2^+}\prod_{j=1}^{N_2}(\beta_{i, j}^{+2})^{\Xi_{+2}^{j,  Y}})\right)\cdots
\\
&\left(k_{+M}^{c_M^+}(\prod_{i=1}^{N_1}(\alpha_i)^{\Xi_{+M}^{i,  X}})^{c_M^+}(\prod_{i=1}^{c_M^+}\prod_{j=1}^{N_2}(\beta_{i, j}^{+M})^{\Xi_{+M}^{j,  Y}})\right)\cdot
\left(k_{-1}^{c_1^-}(\prod_{i=1}^{N_1}(\alpha_i)^{\Xi_{-1}^{i, X}})^{c_1^-}(\prod_{i=1}^{c_1^-}\prod_{j=1}^{N_2}(\beta_{i, j}^{-1})^{\Xi_{-1}^{j,  Y}})\right)
\\
&\left(k_{-2}^{c_2^-}(\prod_{i=1}^{c_2^-}\prod_{j=1}^{N_2}(\beta_{i, j}^{-2})^{\Xi_{-2}^{j,  Y}})\right)\cdots
\left(k_{-M}^{c_M^-}(\prod_{i=1}^{N_1}(\alpha_i)^{\Xi_{-M}^{i,  X}})^{c_M^-}(\prod_{i=1}^{c_M^-}\prod_{j=1}^{N_2}(\beta_{i, j}^{-M})^{\Xi_{-M}^{j,  Y}})\right)+O(V^{-1})\bigg]\cdot
\end{aligned}
\end{equation}
\end{proof}
\begin{thm}
The following limit exists:
\begin{equation}
\begin{aligned}
\mathcal{J}_{\tilde{c}} =& \lim_{V\rightarrow\infty}\frac{\omega_{\tilde{c}}}{V}
\\
=&\sum_{s=1}^N\sum_{
\tiny{
\substack{
     c\in [S]\\c=[y_1, y_2, . . . y_s]}
}
}
\left(
(-1)^{s-1}N_A^s\cdot
\left(k_{+1}^{c_1^+}(\prod_{i=1}^{N_1}(\alpha_i)^{\Xi_{+1}^{i, X}})^{c_1^+}(\prod_{i=1}^{c_1^+}\prod_{j=1}^{N_2}(\beta_{i, j}^{+1})^{\Xi_{+1}^{j,  Y}})\right)
\right.
\\
&
\left(k_{+2}^{c_2^+}(\prod_{i=1}^{c_2^+}\prod_{j=1}^{N_2}(\beta_{i, j}^{+2})^{\Xi_{+2}^{j,  Y}})\right)
\cdots
\left(k_{+M}^{c_M^+}(\prod_{i=1}^{N_1}(\alpha_i)^{\Xi_{+M}^{i,  X}})^{c_M^+}(\prod_{i=1}^{c_M^+}\prod_{j=1}^{N_2}(\beta_{i, j}^{+M})^{\Xi_{+M}^{j,  Y}})\right)
\\
&
\cdot
\left(k_{-1}^{c_1^-}(\prod_{i=1}^{N_1}(\alpha_i)^{\Xi_{-1}^{i, X}})^{c_1^-}(\prod_{i=1}^{c_1^-}\prod_{j=1}^{N_2}(\beta_{i, j}^{-1})^{\Xi_{-1}^{j,  Y}})\right)
\\
&\left(k_{-2}^{c_2^-}(\prod_{i=1}^{c_2^-}\prod_{j=1}^{N_2}(\beta_{i, j}^{-2})^{\Xi_{-2}^{j,  Y}})\right)
\cdots
\left(k_{-M}^{c_M^-}(\prod_{i=1}^{N_1}(\alpha_i)^{\Xi_{-M}^{i,  X}})^{c_M^-}(\prod_{i=1}^{c_M^-}\prod_{j=1}^{N_2}(\beta_{i, j}^{-M})^{\Xi_{-M}^{j,  Y}})\right)
\\
&\cdot\left.\frac{\left|\hat{Q}\left(\left\{y_{1},  y_{2},  \ldots,  y_{s}\right\}^{\mathrm{c}}\right)\right|}{\sum_{j \in S}\left|\hat{Q}\left(\{j\}^{\mathrm{c}}\right)\right|}\right)
\cdot
\prod_{l=1}^M\mathbbm{1}_{\left\{\sum_{h=1}^s
(\mathbbm{1}_{\{y_{h+1}=y_h+\Xi_l^Y\}}-
\mathbbm{1}_{\{y_{h+1}=y_h-\Xi_l^Y\}})=c_l\right\}}.
\end{aligned}
\end{equation}
\begin{remark}
When $V\rightarrow\infty$, it follows from Lemma~\ref{rate} that  
\begin{equation}
    \begin{aligned}
    &r_{+\ell}=V\left(k_{+\ell}N_A\prod_{j=1}^{N_2}(\beta_j)^{\Xi_{+\ell}^{j,  Y}}+O(V^{-1})\right),\quad\ell=2,3,\cdots,M
    \\
    &r_{-\ell}=V\left(k_{-\ell}N_A\prod_{j=1}^{N_2}(\beta_j)^{\Xi_{-\ell}^{j,  Y}}+O(V^{-1})\right),\quad\ell=1,2,\cdots,M-1
\\
or\quad\quad&
\\
    &r_{+1}=V\left(k_{+1}N_A\prod_{i=1}^{N_1}(\alpha_i)^{\Xi_{+1}^{i,  X}}\prod_{j=1}^{N_2}(\beta_j)^{\Xi_{+1}^{j,  Y}}+O(V^{-1})\right)
    \\
    &r_{-M}=V\left(k_{-M}N_A\prod_{i=1}^{N_1}(\alpha_i)^{\Xi_{-M}^{i,  X}}\prod_{j=1}^{N_2}(\beta_j)^{\Xi_{-M}^{j, Y}}+O(V^{-1})\right).
    \end{aligned}
\end{equation}
Denote
\begin{equation}
    \begin{aligned}
    &a_{+\ell}=k_{+\ell}N_A\prod_{j=1}^{N_2}(\beta_j)^{\Xi_{+\ell}^{j,  Y}},\quad\ell=2,3,\cdots,M
    \\
    &a_{-\ell}=k_{-\ell}N_A\prod_{j=1}^{N_2}(\beta_j)^{\Xi_{-\ell}^{j,  Y}},\quad\ell=1,2,\cdots,M-1
    \\
    or\quad\quad&
    \\
    &a_{+1}=k_{+1}N_A\prod_{i=1}^{N_1}(\alpha_i)^{\Xi_{+1}^{i,  X}}\prod_{j=1}^{N_2}(\beta_j)^{\Xi_{+1}^{j,  Y}}
    \\
    &a_{-M}=k_{-M}N_A\prod_{i=1}^{N_1}(\alpha_i)^{\Xi_{-M}^{i,  X}}\prod_{j=1}^{N_2}(\beta_j)^{\Xi_{-M}^{j, Y}}.
    \end{aligned}
\end{equation}
Then
\begin{equation}
    \begin{aligned}
    &r_{+\ell}=V\left(a_{+\ell}+O(V^{-1})\right),\quad\ell=2,3,\cdots,M
    \\
    &r_{-\ell}=V\left(a_{-\ell}+O(V^{-1})\right),\quad\ell=1,2,\cdots,M-1
\\
or\quad\quad&
\\
    &r_{+1}=V\left(a_{+1}+O(V^{-1})\right)
    \\
    &r_{-M}=V\left(a_{-M}+O(V^{-1})\right).
    \end{aligned}
\end{equation}

Replace all $r_{\pm\ell}$ in matrix Q with $a_{\pm\ell}$, there appear a new matrix $\tilde{Q}$. Obviously, $\tilde{Q}$ is computable in our model.

Then, it's easy to verify that $$\lim_{V\rightarrow\infty}\abs{\frac{1}{V}Q\left(\left\{y_{1},  y_{2},  \cdots,  y_{s}\right\}^{\mathrm{c}}\right)}= \abs{\hat{Q}\left(\left\{y_{1},  y_{2},  \cdots,  y_{s}\right\}^{\mathrm{c}}\right)},$$
for any $\left\{y_{1},  y_{2},  \cdots,  y_{s}\right\}\in[S].$
\end{remark}
\end{thm}
\begin{proof}
From equation~\ref{w_c}, we have
$$\omega_{c}=(-1)^{s-1} q_{y_{1} y_{2}} q_{y_{2} y_{3}} \ldots q_{y_{s} y_{1}} \frac{\left|Q\left(\left\{y_{1},  y_{2},  \ldots,  y_{s}\right\}^{\mathrm{c}}\right)\right|}{\sum_{j \in S}\left|Q\left(\{j\}^{\mathrm{c}}\right)\right|}.$$
Then,
    $$\frac{\left|Q\left(\left\{y_{1},  y_{2},  \ldots,  y_{s}\right\}^{\mathrm{c}}\right)\right|}{\sum_{j \in S}\left|Q\left(\{j\}^{\mathrm{c}}\right)\right|}
    =\frac{V^{N-s}\left|\frac{1}{V}Q\left(\left\{y_{1},  y_{2},  \ldots,  y_{s}\right\}^{\mathrm{c}}\right)\right|}{V^{N-1}\sum_{j \in S}\left|\frac{1}{V}Q\left(\{j\}^{\mathrm{c}}\right)\right|}
    =\frac{\left|Q\left(\left\{y_{1},  y_{2},  \ldots,  y_{s}\right\}^{\mathrm{c}}\right)\right|}{V^{s-1}\sum_{j \in S}\left|Q\left(\{j\}^{\mathrm{c}}\right)\right|}$$
From lemma~\ref{q_yy},
\begin{equation}
    \begin{aligned}
&q_{y_{1} y_{2}} q_{y_{2} y_{3}} \ldots q_{y_{s} y_{1}}
\\
=&V^s\bigg[N_A^s\cdot\left(k_{+1}^{c_1^+}(\prod_{i=1}^{N_1}(\alpha_i)^{\Xi_{+1}^{i, X}})^{c_1^+}(\prod_{i=1}^{c_1^+}\prod_{j=1}^{N_2}(\beta_{i, j}^{+1})^{\Xi_{+1}^{j,  Y}})\right)
\left(k_{+2}^{c_2^+}(\prod_{i=1}^{c_2^+}\prod_{j=1}^{N_2}(\beta_{i, j}^{+2})^{\Xi_{+2}^{j,  Y}})\right)\cdots
\\
&\left(k_{+M}^{c_M^+}(\prod_{i=1}^{N_1}(\alpha_i)^{\Xi_{+M}^{i,  X}})^{c_M^+}(\prod_{i=1}^{c_M^+}\prod_{j=1}^{N_2}(\beta_{i, j}^{+M})^{\Xi_{+M}^{j,  Y}})\right)\cdot
\left(k_{-1}^{c_1^-}(\prod_{i=1}^{N_1}(\alpha_i)^{\Xi_{-1}^{i, X}})^{c_1^-}(\prod_{i=1}^{c_1^-}\prod_{j=1}^{N_2}(\beta_{i, j}^{-1})^{\Xi_{-1}^{j,  Y}})\right)
\\
&\left(k_{-2}^{c_2^-}(\prod_{i=1}^{c_2^-}\prod_{j=1}^{N_2}(\beta_{i, j}^{-2})^{\Xi_{-2}^{j,  Y}})\right)\cdots
\left(k_{-M}^{c_M^-}(\prod_{i=1}^{N_1}(\alpha_i)^{\Xi_{-M}^{i,  X}})^{c_M^-}(\prod_{i=1}^{c_M^-}\prod_{j=1}^{N_2}(\beta_{i, j}^{-M})^{\Xi_{-M}^{j,  Y}})\right)+O(V^{-1})\bigg]
    \end{aligned}
\end{equation}
Thus,
\begin{equation}
    \begin{aligned}
    \lim_{V\rightarrow\infty}\frac{\omega_c}{V}
    =&(-1)^{s-1}\lim_{V\rightarrow\infty}\frac{1}{V}\cdot
    \frac{\left|Q\left(\left\{y_{1},  y_{2},  \ldots,  y_{s}\right\}^{\mathrm{c}}\right)\right|}{V^{s-1}\sum_{j \in S}\left|Q\left(\{j\}^{\mathrm{c}}\right)\right|}\cdot V^s\\ 
    &\cdot\bigg[N_A^s\cdot\left(k_{+1}^{c_1^+}(\prod_{i=1}^{N_1}(\alpha_i)^{\Xi_{+1}^{i, X}})^{c_1^+}(\prod_{i=1}^{c_1^+}\prod_{j=1}^{N_2}(\beta_{i, j}^{+1})^{\Xi_{+1}^{j,  Y}})\right)
\left(k_{+2}^{c_2^+}(\prod_{i=1}^{c_2^+}\prod_{j=1}^{N_2}(\beta_{i, j}^{+2})^{\Xi_{+2}^{j,  Y}})\right)\cdots
\\
&\left(k_{+M}^{c_M^+}(\prod_{i=1}^{N_1}(\alpha_i)^{\Xi_{+M}^{i,  X}})^{c_M^+}(\prod_{i=1}^{c_M^+}\prod_{j=1}^{N_2}(\beta_{i, j}^{+M})^{\Xi_{+M}^{j,  Y}})\right)\cdot
\left(k_{-1}^{c_1^-}(\prod_{i=1}^{N_1}(\alpha_i)^{\Xi_{-1}^{i, X}})^{c_1^-}(\prod_{i=1}^{c_1^-}\prod_{j=1}^{N_2}(\beta_{i, j}^{-1})^{\Xi_{-1}^{j,  Y}})\right)
\\
&\left(k_{-2}^{c_2^-}(\prod_{i=1}^{c_2^-}\prod_{j=1}^{N_2}(\beta_{i, j}^{-2})^{\Xi_{-2}^{j,  Y}})\right)\cdots
\left(k_{-M}^{c_M^-}(\prod_{i=1}^{N_1}(\alpha_i)^{\Xi_{-M}^{i,  X}})^{c_M^-}(\prod_{i=1}^{c_M^-}\prod_{j=1}^{N_2}(\beta_{i, j}^{-M})^{\Xi_{-M}^{j,  Y}})\right)+O(V^{-1})\bigg]
\\
=&(-1)^{s-1}\bigg[N_A^s\cdot\left(k_{+1}^{c_1^+}(\prod_{i=1}^{N_1}(\alpha_i)^{\Xi_{+1}^{i, X}})^{c_1^+}(\prod_{i=1}^{c_1^+}\prod_{j=1}^{N_2}(\beta_{i, j}^{+1})^{\Xi_{+1}^{j,  Y}})\right)
\left(k_{+2}^{c_2^+}(\prod_{i=1}^{c_2^+}\prod_{j=1}^{N_2}(\beta_{i, j}^{+2})^{\Xi_{+2}^{j,  Y}})\right)\cdots
\\
&\left(k_{+M}^{c_M^+}(\prod_{i=1}^{N_1}(\alpha_i)^{\Xi_{+M}^{i,  X}})^{c_M^+}(\prod_{i=1}^{c_M^+}\prod_{j=1}^{N_2}(\beta_{i, j}^{+M})^{\Xi_{+M}^{j,  Y}})\right)\cdot
\left(k_{-1}^{c_1^-}(\prod_{i=1}^{N_1}(\alpha_i)^{\Xi_{-1}^{i, X}})^{c_1^-}(\prod_{i=1}^{c_1^-}\prod_{j=1}^{N_2}(\beta_{i, j}^{-1})^{\Xi_{-1}^{j,  Y}})\right)
\\
&\left(k_{-2}^{c_2^-}(\prod_{i=1}^{c_2^-}\prod_{j=1}^{N_2}(\beta_{i, j}^{-2})^{\Xi_{-2}^{j,  Y}})\right)\cdots
\left(k_{-M}^{c_M^-}(\prod_{i=1}^{N_1}(\alpha_i)^{\Xi_{-M}^{i,  X}})^{c_M^-}(\prod_{i=1}^{c_M^-}\prod_{j=1}^{N_2}(\beta_{i, j}^{-M})^{\Xi_{-M}^{j,  Y}})\right)\bigg]
\\
&\cdot\frac{\left|\hat{Q}\left(\left\{y_{1},  y_{2},  \ldots,  y_{s}\right\}^{\mathrm{c}}\right)\right|}{\sum_{j \in S}\left|\hat{Q}\left(\{j\}^{\mathrm{c}}\right)\right|}
    \end{aligned}
\end{equation}
Therefore,
\begin{align*}
\mathcal{J}_{\tilde{c}}=&\lim_{V\rightarrow\infty}\frac{1}{V}\sum_{s=1}^N\sum_{
\tiny{
\substack{
     c\in [S]\\c=[y_1, y_2, . . . y_s]}
}
}
\left(
w_c
\cdot
\prod_{l=1}^M\mathbbm{1}_{\left\{\sum_{h=1}^s
(\mathbbm{1}_{\{y_{h+1}=y_h+\Xi_l^Y\}}-
\mathbbm{1}_{\{y_{h+1}=y_h-\Xi_l^Y\}})=c_l\right\}}
\right)
\\
=&\sum_{s=1}^N\sum_{
\tiny{
\substack{
     c\in [S]\\c=[y_1, y_2, . . . y_s]}
}
}
\left(
\lim_{V\rightarrow\infty}\frac{w_c}{V}
\cdot
\prod_{l=1}^M\mathbbm{1}_{\left\{\sum_{h=1}^s
(\mathbbm{1}_{\{y_{h+1}=y_h+\Xi_l^Y\}}-
\mathbbm{1}_{\{y_{h+1}=y_h-\Xi_l^Y\}})=c_l\right\}}
\right)
\\
=&\sum_{s=1}^N\sum_{
\tiny{
\substack{
     c\in [S]\\c=[y_1, y_2, . . . y_s]}
}
}
\left(
(-1)^{s-1}N_A^s\cdot
\left(k_{+1}^{c_1^+}(\prod_{i=1}^{N_1}(\alpha_i)^{\Xi_{+1}^{i, X}})^{c_1^+}(\prod_{i=1}^{c_1^+}\prod_{j=1}^{N_2}(\beta_{i, j}^{+1})^{\Xi_{+1}^{j,  Y}})\right)
\right.
\\
&
\left(k_{+2}^{c_2^+}(\prod_{i=1}^{c_2^+}\prod_{j=1}^{N_2}(\beta_{i, j}^{+2})^{\Xi_{+2}^{j,  Y}})\right)
\cdots
\left(k_{+M}^{c_M^+}(\prod_{i=1}^{N_1}(\alpha_i)^{\Xi_{+M}^{i,  X}})^{c_M^+}(\prod_{i=1}^{c_M^+}\prod_{j=1}^{N_2}(\beta_{i, j}^{+M})^{\Xi_{+M}^{j,  Y}})\right)
\\
&
\cdot
\left(k_{-1}^{c_1^-}(\prod_{i=1}^{N_1}(\alpha_i)^{\Xi_{-1}^{i, X}})^{c_1^-}(\prod_{i=1}^{c_1^-}\prod_{j=1}^{N_2}(\beta_{i, j}^{-1})^{\Xi_{-1}^{j,  Y}})\right)
\\
&\left(k_{-2}^{c_2^-}(\prod_{i=1}^{c_2^-}\prod_{j=1}^{N_2}(\beta_{i, j}^{-2})^{\Xi_{-2}^{j,  Y}})\right)
\cdots
\left(k_{-M}^{c_M^-}(\prod_{i=1}^{N_1}(\alpha_i)^{\Xi_{-M}^{i,  X}})^{c_M^-}(\prod_{i=1}^{c_M^-}\prod_{j=1}^{N_2}(\beta_{i, j}^{-M})^{\Xi_{-M}^{j,  Y}})\right)
\\
&\cdot\left.\frac{\left|\hat{Q}\left(\left\{y_{1},  y_{2},  \ldots,  y_{s}\right\}^{\mathrm{c}}\right)\right|}{\sum_{j \in S}\left|\hat{Q}\left(\{j\}^{\mathrm{c}}\right)\right|}\right)
\cdot
\prod_{l=1}^M\mathbbm{1}_{\left\{\sum_{h=1}^s
(\mathbbm{1}_{\{y_{h+1}=y_h+\Xi_l^Y\}}-
\mathbbm{1}_{\{y_{h+1}=y_h-\Xi_l^Y\}})=c_l\right\}}.
    \end{align*}
\end{proof}
\section*{Acknowledgements}
This work is partially supported by the National Natural Science Foundation of China (No. 11701265, No. 11961033).


\end{document}